\documentclass[12pt]{amsart}
\usepackage[margin=1.2in]{geometry}
\usepackage{graphicx,latexsym}
\usepackage{amsfonts, amssymb, amsmath, amsthm}
\usepackage{tikz-cd}
\usepackage{tikz}
\usepackage{amsfonts, amssymb, amsmath, amsthm, bm}
\usepackage{verbatim}

\newtheorem{theorem}{Theorem}[section]
\newtheorem{proposition}[theorem]{Proposition}
\newtheorem{lemma}[theorem]{Lemma}
\newtheorem{corollary}[theorem]{Corollary}

\theoremstyle{definition}

\theoremstyle{remark}
\newtheorem{remark}[theorem]{Remark}

 \numberwithin{equation}{section}

\newcommand{\N}{\mathbb N}

\newcommand{\C}{\mathbb C}
\newcommand{\R}{\mathbb R}

\newcommand{\dx}{{\rm d}x }
\newcommand{\dxi}{{\rm d}\xi }
\newcommand{\dt}{{\rm d}t }

\usetikzlibrary{arrows,chains,matrix,positioning,scopes}
\makeatletter
\tikzset{join/.code=\tikzset{after node path={%
\ifx\tikzchainprevious\pgfutil@empty\else(\tikzchainprevious)%
edge[every join]#1(\tikzchaincurrent)\fi}}}
\makeatother
\tikzset{>=stealth',every on chain/.append style={join},
         every join/.style={->}}
\tikzstyle{labeled}=[execute at begin node=$\scriptstyle,
   execute at end node=$]
\begin{document}
\title[nonlinear theory of infrahyperfunctions]{A nonlinear theory of infrahyperfunctions}
\author[A. Debrouwere]{Andreas Debrouwere}
\address{Department of Mathematics: Analysis, Logic and Discrete Mathematics, Ghent University, Krijgslaan 281, 9000 Gent, Belgium}
\email{Andreas.Debrouwere@UGent.be}
\thanks{A. Debrouwere gratefully acknowledges support by Ghent University, through a BOF Ph.D.-grant.}

\author[H. Vernaeve]{Hans Vernaeve}
\address{Department of Mathematics: Analysis, Logic and Discrete Mathematics, Ghent University, Krijgslaan 281, 9000 Gent, Belgium}
\email{Hans.Vernaeve@UGent.be}
\thanks{H. Vernaeve was supported by grant 1.5.138.13N of the Research Foundation Flanders FWO}

\author[J. Vindas]{Jasson Vindas}
\thanks{J. Vindas was supported by Ghent University, through the BOF-grants 01N01014 and 01J11615.}
\address{Department of Mathematics: Analysis, Logic and Discrete Mathematics, Ghent University, Krijgslaan 281, 9000 Gent, Belgium}
\email{Jasson.Vindas@UGent.be}
\begin{abstract}We develop a nonlinear theory for infrahyperfunctions (also referred to
as \emph{quasianalytic (ultra)distributions} by L.\ H\"{o}rmander). In the hyperfunction case our work can be summarized as follows. We construct a differential algebra that contains the space of hyperfunctions as a linear differential subspace and in which the multiplication of real analytic functions coincides with their ordinary product. Moreover, by proving an analogue of Schwartz's impossibility result for hyperfunctions, we show that this embedding is optimal. Our results fully solve an earlier question raised by M.\ Oberguggenberger.
\end{abstract}
\maketitle
\section{Introduction}
The nonlinear theory of generalized functions was initiated by J.-F.\ Colombeau \cite{Colombeau84, Colombeau85} who constructed  a differential algebra containing the space of distributions as a linear differential subspace and the space of smooth functions as a subalgebra. Algebras of generalized functions provide a framework for linear equations with strongly singular data and nonlinear equations \cite{Hormann-DeHoop,Hormann-Ober-Pilipovic}. 
We refer to the monograph \cite{Ober} for many interesting applications to the theory of partial differential equations.

On the other hand, for several natural linear problems the space of distributions is not the suitable setting, e.g.\ Cauchy problems for weakly hyperbolic equations -- even with smooth coefficients -- are in general not well-posed in the space of distributions \cite{Colombini83,Colombini82}.
Such considerations motivated the search for and study of new spaces of linear generalized functions 
like ultradistributions \cite{Komatsu} and hyperfunctions \cite{Sato, Morimoto}. For instance, under suitable conditions the above Cauchy problems become well-posed in spaces of ultradistributions \cite{Colombini83,Colombini03,Garetto}, while the space of hyperfunctions is the convenient setting for the treatment of partial differential equations with real analytic coefficients \cite{Bony, Komatsu71}.

In \cite{DVV} we presented a nonlinear theory of non-quasianalytic ultradistributions (cf.\ \cite{Delcroix, Gramchev}). The goal of this paper is to further extend these results  and develop a nonlinear theory for hyperfunctions. In fact, we consider general spaces of infrahyperfunctions of class $\{M_p\}$ \cite{Hormander} for a quasianalytic weight sequence $M_p$ \cite{Komatsu} and construct a differential algebra that contains the space of infrahyperfunctions of class $\{M_p\}$ as a differential subspace and in which the multiplication of quasianalytic functions of class $\{M_p\}$ coincides with their pointwise product. Moreover, by establishing a Schwartz impossibility type result for infrahyperfunctions, we show that our embeddings are optimal in the sense of being consistent with the pointwise multiplication of ordinary functions. The case $M_p = p!$ corresponds to the hyperfunction case, thereby fully settling a question of M.\ Oberguggenberger  \cite[p.\ 286, Prob.\ 27.2]{Ober}. We mention the pioneer article \cite{Pil2005} by S. Pilipovi\'{c}, where a so-called algebra of megafunctions in dimension 1 was introduced.
 We also believe that our work puts the notion of `very weak solution' introduced by C.\ Garetto and M.\ Ruzhansky in their study on well-posedness of weakly hyperbolic equations with time dependent nonregular coefficients \cite{garetto-r2015} in a broader perspective and could possibly lead to a natural framework for extensions of their results.

This paper is organized as follows. In the preliminary Section \ref{preliminaries} we   collect some useful properties of the spaces of quasianalytic functions and their duals (quasianalytic functionals) that will be used later on in the article. Sheaves of infrahyperfunctions are introduced in Section \ref{sect-infrahyperfunctions}. They were first constructed by L.\ H\"ormander\footnote{He uses the name quasianalytic distributions in \cite{Hormander}; the terminology infrahyperfunctions comes from \cite{de Roever, Petzsche84}.} \cite{Hormander} but we give a short proof of their existence based on H\"{o}rmander's support theorem for quasianalytic functionals  and a general method for the construction of flabby sheaves due to K.\ Junker and Y.\ Ito \cite{Ito, Junker}. In Section \ref{section algebra} we define our algebras of generalized functions, provide a null-characterization for the space of negligible sequences, and show that the totality of our algebras forms a sheaf on $\R^d$. Due to the absence of nontrivial compactly supported quasianalytic functions, the latter is much more difficult to establish than for Colombeau algebras of non-quasianalytic type. The proof is based on some of our recent results on the solvability of the Cousin problem for vector-valued quasianalytic  functions \cite{DV}. We mention that the motivation for \cite{DV} was precisely this problem. In Section \ref{section-embedding} we embed the infrahyperfunctions of class $\{M_p\}$  into our algebras and show that the multiplication of ultradifferentiable functions of class $\{M_p\}$ is preserved under the embedding we shall construct. Finally, we present a Schwartz impossibility type result for infrahyperfunctions which shows that our embeddings are optimal.
 
\section{Spaces of quasianalytic functions and their duals}\label{preliminaries}
In this section we fix the notation and introduce spaces of quasianalytic functions and their duals, the so-called spaces of quasianalytic functionals. We also discuss the notion of support for quasianalytic functionals and state a result about the solvability of the Cousin problem for vector-valued quasianalytic functions \cite{DV} that will be used in Section \ref{section algebra}.

Set $\N_0 =  \N \cup \{0\}$. Let $(M_p)_{p \in \N_0}$ be a weight sequence of positive real numbers and set $m_p: = M_p/ M_{p-1}$, $p \in \N$. We shall always assume that $M_0=1$ and that $m_p \rightarrow \infty$ as $p \rightarrow \infty$. Furthermore, we will make use of various of the following conditions:
\begin{itemize}
\item [$(M.1)$\:] $M^{2}_{p}\leq M_{p-1}M_{p+1},$  $p\in \N$,
\item [$(M.2)'$]$M_{p+1}\leq A H^{p+1} M_p$, $p\in\N_0$, for some $A,H \geq 1$,
\item [$(M.2)$\:]$M_{p+q}\leq A H^{p+q} M_pM_q$, $p,q\in\N_0$, for some $A,H \geq 1$,
\item[$(M.2)^\ast$] $2m_p \leq Cm_{pQ}$, $p \in \mathbb{N}$, for some $Q \in \N$, $C > 0$,
\item [$(QA)$\:] $\displaystyle \sum_{p=1}^{\infty}\frac{1}{m_{p}}=\infty, $
\item[$(NE)$\:] $p!\leq Ch^{p}M_{p},$ $p\in\N_0$, for some $C,h > 0$.
\end{itemize}
We write $M_\alpha = M_{|\alpha|}$ for $\alpha \in \N_0^d$. The associated function of $M_p$ is defined as
$$
M(t):=\sup_{p\in\N_0}\log\frac{t^p}{M_p},\qquad t > 0,
$$
and $M(0)=0$. We define $M$ on $\R^d$ as the radial function $M(x) = M(|x|)$, $x \in \R^d$. As usual, the relation $M_p\subset N_p$ between two weight sequences means that there are $C,h>0$ such that 
$M_p\leq Ch^{p}N_{p},$ $p\in\N_0$. The stronger relation $M_p\prec N_p$ means that the latter inequality remains valid for every $h>0$ and a suitable $C=C_{h}>0$. 
Condition $(NE)$ can therefore be formulated as $p! \subset M_p$. We refer to \cite{Komatsu, B-M-M} for the meaning of these conditions and their translation in terms of the associated function. In particular, under $(M.1)$, the assumption $(M.2)$ holds \cite[Prop.\ 3.6]{Komatsu} if and only if
$$
2M(t) \leq M(Ht) + \log A, \qquad  t \geq 0.
$$
Suppose $K$ is a regular compact subset of $\R^d$, that is, $\overline{\operatorname{int} K} = K$. For $h > 0$ we write $\mathcal{E}^{M_p,h}(K)$ for the Banach space of all $\varphi \in C^\infty(K)$ such that
$$
\| \varphi \|_{K,h} := \sup_{\alpha \in \N_0^d}\sup_{x \in K} \frac{|\varphi^{(\alpha)}(x)|}{h^{|\alpha|}M_{\alpha}} < \infty. 
$$
For an open set $\Omega$ in $\R^d$, we define
$$
\mathcal{E}^{\{M_p\}}(\Omega) = \varprojlim_{K \Subset \Omega} \varinjlim_{h \rightarrow \infty}\mathcal{E}^{M_p,h}(K).
$$
The elements of $\mathcal{E}^{\{M_p\}}(\Omega)$ are called \emph{ultradifferentiable functions of class $\{M_p\}$ (or Roumieu type)} \emph{in} $\Omega$. When $M_p = p!$ the space $\mathcal{E}^{\{M_p\}}(\Omega)$ coincides with the space $\mathcal{A}(\Omega)$ of real analytic functions in $\Omega$. By the Denjoy-Carleman theorem, $(QA)$ means that there are no nontrivial compactly supported ultradifferentiable functions of class $\{M_p\}$, or equivalently, that a function $\varphi \in \mathcal{E}^{\{M_p\}}(\Omega)$ which vanishes on an open subset that meets every connected component of $\Omega$, is necessarily identically zero. 

It should be mentioned that, under $(M.1)$, $(M.2)$ and $(NE)$, a weight sequence $M_p$ satisfies $(M.2)^*$ \cite[Thm.\ 14]{B-M-M} if and only if $\omega = M$ is a Braun-Meise-Taylor type weight function  \cite{B-M-T} (as defined in \cite[p.\ 426]{B-M-M}). In such a case, we have $\mathcal{E}^{\{M_p\}}(\Omega) = \mathcal{E}_{\{\omega\}}(\Omega)$ as locally convex spaces.

We write $\mathcal{R}$ for the family of all positive real sequences $(r_j)_{j \in \N_0}$ with $r_0 =1$ that increase (not necessarily strictly) to infinity. This set is partially ordered and directed by the relation $r_j \preceq s_j$, which means that there is a $j_0 \in \N_0$ such that $r_j \leq s_j$ for all $j \geq j_0$. Following Komatsu \cite{Komatsu3}, we removed the non-quasianalyticity assumption in his projective description of $\mathcal{E}^{\{M_p\}}(\Omega)$  and we found in \cite{DV} an explicit system of seminorms generating the topology of $\mathcal{E}^{\{M_p\}}(\Omega)$ also in the quasianalytic case.

\begin{lemma}\emph{(\cite[Prop.\ 4.8]{DV})}\label{systemnorms}    
Let $M_p$ be a weight function satisfying $(M.1)$, $(M.2)'$, and $(NE)$. A function $\varphi \in C^\infty(\Omega)$ belongs to $\mathcal{E}^{\{M_p\}}(\Omega)$ if and only if 
$$
\| \varphi \|_{K, r_j} := \sup_{\alpha \in \N_0^d}\sup_{x \in K} \frac{|\varphi^{(\alpha)}(x)|}{M_{\alpha}\prod_{j = 0}^{|\alpha|}r_j} < \infty,
$$
for all $K \Subset \Omega$ and $r_j \in \mathcal{R}$. Moreover, the topology of $\mathcal{E}^{\{M_p\}}(\Omega)$ is generated by the system of seminorms $\{\| \, \|_{K, r_j} \, : \, K \Subset \Omega, r_j \in \mathcal{R}\}$.
\end{lemma}

Suppose that the weight sequence $M_p$ satisfies $(M.1)$, $(M.2)'$, $(QA)$, and $(NE)$.  The elements of the dual space $\mathcal{E}'^{\{M_p\}}(\Omega)$ are called \emph{quasianalytic functionals of class $\{M_p\}$ (or Roumieu type)} \emph{in} $\Omega$, while those of $\mathcal{A}'(\Omega)$ are called \emph{analytic functionals in} $\Omega$. The properties $(M.1)$, $(M.2)'$, and $(NE)$ imply that the space of entire functions is dense in $\mathcal{E}^{\{M_p\}}(\Omega)$ \cite[Prop 3.2]{Hormander}. Hence, for any $\Omega' \subseteq \Omega$ and any other weight
sequence $N_p$ with $p! \subset N_p \subset M_p$ we may identify $\mathcal{E}'^{\{M_p\}}(\Omega')$ with a subspace of $\mathcal{E}'^{\{N_p\}}(\Omega)$. In particular, we always have
$\mathcal{E}'^{\{M_p\}}(\Omega') \subseteq  \mathcal{A}'(\Omega)$.

An \emph{ultradifferential operator of class} $\{M_p\}$ is an infinite order differential operator
$$ 
P(D) = \sum_{\alpha \in \N_0^d} a_\alpha D^\alpha, \qquad a_\alpha \in \C,
$$
($D^{\alpha}=(-i \partial)^{\alpha}$) where the coefficients satisfy the estimate 
$$
|a_\alpha| \leq \frac{CL^{|\alpha|}}{M_{|\alpha|} }
$$
for every $L>0$ and some $C=C_{L} > 0$. If $M_p$ satisfies $(M.2)$, then $P(D)$ acts continuously on $\mathcal{E}^{\{M_p\}}(\Omega)$ and hence it can be defined by duality on $\mathcal{E}'^{\{M_p\}}(\Omega)$. 

Next, we discuss the notion of support for quasianalytic functionals. For a compact $K$ in $\R^d$, we define the space of germs of ultradifferentiable functions on $K$ as
$$
\mathcal{E}^{\{M_p\}}[K] = \varinjlim_{K \Subset \Omega} \mathcal{E}^{\{M_p\}}(\Omega),
$$ 
a $(DFS)$-space.
Since
$$
\mathcal{E}^{\{M_p\}}(\R^d) \cong \varprojlim_{K \Subset \R^d} \mathcal{E}^{\{M_p\}}[K]
$$
as locally convex spaces and $\mathcal{E}^{\{M_p\}}(\R^d)$ is dense in each $\mathcal{E}^{\{M_p\}}[K]$, we have the following isomorphism of locally convex spaces
$$
\mathcal{E}'^{\{M_p\}}(\R^d) \cong \varinjlim_{K \Subset \R^d} \mathcal{E}'^{\{M_p\}}[K].
$$
Let $f \in \mathcal{E}'^{\{M_p\}}(\R^d)$. A compact $K \Subset \R^d$ is said to be an $\{M_p\}$-\emph{carrier of} 
$f$ if $f \in \mathcal{E}'^{\{M_p\}}[K]$. It is well known that for every $f \in \mathcal{A}'(\Omega)$ 
there is a smallest compact $K \Subset \R^d$ among the $\{p!\}$-carriers of $f$, called the \emph{support of} $f$ and denoted by $\operatorname{supp}
_{\mathcal{A}'}f$. This essentially follows from the cohomology of the sheaf of germs of analytic functions (see e.g. \cite{Morimoto}). An elementary proof based on 
the properties of the Poisson transform of analytic functionals is given in \cite[Sect 9.1]{Hormander2}. For a proof based on the heat kernel method see
\cite{Matsuzawa}. H\"ormander noticed that a similar result holds for quasianalytic functionals of Roumieu type. More precisely, he showed the following important result:
\begin{proposition}\emph{(\cite[Cor.\ 3.5]{Hormander})}\label{support}
Let $M_p$ be a weight sequence satisfying $(M.1)$, $(M.2)'$, $(QA)$ and $(NE)$. For every $f \in \mathcal{E}'^{\{M_p\}}(\Omega)$ there is a smallest compact set among the $\{M_p\}$-carriers of $f$ and that set coincides with $\operatorname{supp}_{\mathcal{A}'}f$. We simply denote this set by $\operatorname{supp} f$.
\end{proposition}

Finally, we introduce vector-valued quasianalytic functions and state a sufficient condition on the target space for the solvability of the Cousin problem. Let $F$ be a locally convex space. We write $\mathcal{E}^{\{M_p\}}(\Omega; F)$ for the space of all $\bm{\varphi} \in C^\infty(\Omega ; F)$ such that for each continuous seminorm $q$ on $F$, $K \Subset \Omega$, and $r_j \in \mathcal{R}$ it holds that
$$
q_{K,r_j}(\bm{\varphi}) := \sup_{\alpha \in \N_0^d}\sup_{x \in K} \frac{q(\bm{\varphi}^{(\alpha)}(x))}{M_{\alpha}\prod_{j = 0}^{|\alpha|}r_j} < \infty. 
$$
We endow it with the locally convex topology generated by the system of seminorms $\{q_{K,r_j} \, : \, q \mbox{ continuous seminorm on $F$}, K \Subset \Omega, r_j \in \mathcal{R} \}$. 

A Fr\'echet space $E$ with a generating system of seminorms $\{ \| \, \|_k \, : \, k \in \N\}$ has the property $(\underline{DN})$ \cite[p.\ 368]{Meise} if 
$$
(\exists m \in \N)(\forall k \in \N) (\exists j \in \N)(\exists \tau \in (0,1))( \exists C >0)
$$
$$
\|y\|_k \leq C \|y\|^{1-\tau}_m \|y\|^{\tau}_j, \qquad y \in E.
$$
We use the notation  $F'_\beta$ for $F'$ endowed with the strong topology. 
\begin{proposition}\emph{(\cite[Thm.\ 6.7]{DV})}
\label{Cousin}
Let $M_p$ be a weight sequence satisfying $(M.1)$, $(M.2)'$, $(QA)$, and $(NE)$. Let $\Omega \subseteq \R^d$ be open, let $\mathcal{M} = \{ \Omega_i \, : \, i \in I\}$ be an open covering of $\Omega$, and let $F$ be a $(DFS)$-space such that $F'_\beta$ has the property $(\underline{DN})$. Suppose $\bm{\varphi}_{i,j} \in \mathcal{E}^{\{M_p\}}(\Omega_i \cap \Omega_j;F)$, $i,j \in I$, are given $F$-valued functions such that
$$
\bm{\varphi}_{i,j} + \bm{\varphi}_{j,k} + \bm{\varphi}_{k,i} = 0 \qquad \mbox{on }  \Omega_i \cap \Omega_j \cap \Omega_k, 
$$
for all $i,j,k \in I$.Then, there are $\bm{\varphi}_{i} \in \mathcal{E}^{\{M_p\}}(\Omega_i;F)$, $i \in I$, such that
$$
\bm{\varphi}_{i,j} = \bm{\varphi}_{i} - \bm{\varphi}_{j} \qquad \mbox{on }  \Omega_i \cap \Omega_j,
$$
for all $i,j \in I$.
\end{proposition}

Naturally, rephrased in the language of cohomology groups \cite{Morimoto} Proposition \ref{Cousin} says that the first cohomology group of the open covering $\mathcal{M}$  with coefficients in the sheaf of $F$-valued quasianalytic functions $\mathcal{E}^{\{M_p\}}(\: \cdot \:;F)$ vanishes, that is, $H^{1}(\mathcal{M},\mathcal{E}^{\{M_p\}}(\ \cdot \ ;F))=0$.
  
\section{Sheaves of infrahyperfunctions}\label{sect-infrahyperfunctions}
In  \cite{Hormander} H\"ormander constructed a flabby sheaf $\mathcal{B}^{\{M_p\}}$ with the property that its set of sections with support in $K$, $K \Subset \R^d$, coincides with the space of quasianalytic functionals of class $\{M_p\}$ supported in $K$. In this section, we give a short proof of H\"ormander's result and discuss some of the basic properties of the sheaf $\mathcal{B}^{\{M_p\}}$.

Let $X$ be a topological space and let $\mathcal{F}$ be a sheaf on $X$. For $U \subseteq X$ open  and $A \subseteq U$ we write $\Gamma_A(U, \mathcal{F})$ for the set of sections over $U$ with support in $A$. Define
$$
\Gamma_c(U,\mathcal{F}) = \bigcup_{K \Subset U} \Gamma_K(U, \mathcal{F}).
$$

\begin{proposition}\emph{(\cite[Sect.\ 6]{Hormander})}\label{constructioninfra}
Let $M_p$ be a weight sequence satisfying $(M.1)$, $(M.2)'$, $(QA)$, and $(NE)$. Then, there exists an (up to sheaf isomorphism) unique flabby sheaf $\mathcal{B}^{\{M_p\}}$  over $\R^d$ such that  
$$
\Gamma_K(\R^d, \mathcal{B}^{\{M_p\}}) = \mathcal{E}'^{\{M_p\}}[K], \qquad  K \Subset \R^d. 
$$
Moreover, for every relatively compact open subset $\Omega$ of $\R^d$, one has  
$$
\mathcal{B}^{\{M_p\}}(\Omega) = \mathcal{E}'^{\{M_p\}}[\overline{\Omega}] / \mathcal{E}'^{\{M_p\}}[\partial \Omega].
$$
\end{proposition}
We call $\mathcal{B}^{\{M_p\}}$  the \emph{sheaf of infrahyperfunctions of class $\{M_p\}$ (or Roumieu type).} For $M_p = p!$ this is exactly the sheaf of hyperfunctions $\mathcal{B}$. 
Our proof of Proposition \ref{constructioninfra} relies on the following general method for the construction of flabby sheaves with prescribed compactly supported sections, due to Junker and Ito \cite{Ito, Junker} (they used it to construct sheaves of vector-valued (Fourier) hyperfunctions). The idea goes back to Martineau's duality approach to hyperfunctions \cite{Martineau, Schapira}.  
\begin{lemma}\emph{(\cite[Thm. 1.2]{Ito})}\label{Junkertheorem}
Let $X$ be a second countable, locally compact topological space. Assume that for each compact $K \Subset X$ a Fr\'echet space $F_K$ is given and that for each two compacts $K_1, K_2 \Subset X$,  $K_1 \subseteq K_2$, there is an injective linear continuous mapping $\iota_{K_2,K_1}: F_{K_1} \rightarrow F_{K_2}$ such that for $K_1 \subseteq K_2 \subseteq K_3$, $\iota_{K_3,K_1} = \iota_{K_3,K_2} \circ \iota_{K_2,K_1}$ and $\iota_{K_1,K_1} = \operatorname{id}$. We shall identify $F_{K_1}$ with its image under  the mapping $\iota_{K_1,K_2}$. Suppose that the following conditions are satisfied:
\begin{enumerate}
\item[$(FS1)$] If $K_1, K_2 \Subset X$, $K_1 \subseteq K_2$, have the property that every connected component of $K_2$ meets $K_1$, then $F_{K_1}$ is dense in $F_{K_2}$.
\item[$(FS2)$] For $K_1, K_2 \Subset X$ the mapping 
$$ F_{K_1} \times F_{K_2} \rightarrow F_{K_1 \cup K_2}: (f_1,f_2) \rightarrow f_1 +  f_2$$
is surjective.
\item[$(FS3)$] $(i)$ For $K_1, K_2 \Subset X$ it holds that 
$$F_{K_1 \cap K_2} = F_{K_1} \cap F_{K_2}.$$
$(ii)$ Let $K_1 \supseteq K_2 \supseteq \ldots$ be a decreasing sequence  of compacts in $X$ and set $K = \cap_n K_n$. Then, 
$$ F_K =  \bigcap_{n \in \N}F_{K_n}.$$
\item[$(FS4)$] $F_\emptyset = \{0\}$.
\end{enumerate}
Then, there exists an (up to sheaf isomorphism) unique flabby sheaf $\mathcal{F}$ over $X$ such that
$$
\Gamma_K(X,\mathcal{F}) = F_K, \qquad K \Subset X.
$$
 Moreover, for every relatively compact open subset $U$ of $X$, one has
$$
\mathcal{F}(U) = F_{\overline{U}} / F_{\partial U}.
$$
\end{lemma} 
\begin{proof}[Proof of Proposition \ref{constructioninfra}]  
We use Lemma \ref{Junkertheorem}. Set $X = \R^d$, $F_K = \mathcal{E}'^{\{M_p\}}_\beta[K]$, and $\iota_{K_2, K_1} = {}^tr_{K_2,K_1}$ with $r_{K_2,K_1}$ the canonical restriction mapping 
$$
\mathcal{E}^{\{M_p\}}[K_2] \rightarrow \mathcal{E}^{\{M_p\}}[K_1]. 
$$
Condition $(FS1)$ is a consequence of $(QA)$ while $(FS3)$ and $(FS4)$ are satisfied because of Proposition 
\ref{support}. By the well known criterion for surjectivity of continuous linear mappings between Fr\'echet spaces \cite[Thm.\ 37.2]{Treves}, $(FS2)$ follows from the fact that the mapping
$$
 \mathcal{E}^{\{M_p\}}[K_1 \cup K_2] \rightarrow  \mathcal{E}^{\{M_p\}}[K_1] \times  \mathcal{E}^{\{M_p\}}[K_2]: \varphi \rightarrow (r_{K_1 \cup K_2, K_1}(\varphi), r_{K_1 \cup K_2, K_2}(\varphi) )
$$
is injective and has closed range.
\end{proof}
We shall often use the following extension principle.
\begin{lemma}\emph{(\cite[Lemma 2.3, p. 226]{Komatsu2})}\label{extension}
Let $X$ be a second countable topological space and let $\mathcal{F}$ and $\mathcal{G}$ be soft sheaves on $X$. Let $\rho_c: \Gamma_c(X, \mathcal{F}) \rightarrow \Gamma_c(X, \mathcal{G})$ be a linear mapping such that 
$$
\operatorname{supp} \rho_c (T)  \subseteq \operatorname{supp} T, \qquad T \in \Gamma_c(X, \mathcal{F}).
$$
Then, there is a unique sheaf morphism $\rho: \mathcal{F} \rightarrow \mathcal{G}$ such that, for every open set $U$ in $X$, we have $\rho_U(T) = \rho_c(T)$ for all $T  \in \Gamma_c(U, \mathcal{F})$. If, moreover, 
$$
\operatorname{supp} \rho_c (T) = \operatorname{supp} T, \qquad T \in \Gamma_c(X, \mathcal{F}),
$$
then $\rho$ is injective.
\end{lemma}
\begin{proposition}\label{propertiesinfra} Let $M_p$ be a weight sequence satisfying $(M.1)$, $(M.2)'$, $(QA)$ and $(NE)$.
\begin{enumerate} 
\item[$(i)$] Let $N_p$ be another weight sequence satisfying $(M.1)$, $(M.2)'$, $(QA)$ and $(NE)$. If $N_p \subset M_p$, then $\mathcal{B}^{\{M_p\}}$ is a subsheaf of
$\mathcal{B}^{\{N_p\}}$.  In particular, $\mathcal{B}^{\{M_p\}}$ is always a subsheaf of the sheaf of hyperfunctions $\mathcal{B}$. 
\item[$(ii)$] The sheaf of distributions $\mathcal{D}'$ is a subsheaf of $\mathcal{B}^{\{M_p\}}$ and the following diagram commutes
\[
\begin{tikzcd}
\mathcal{D}' \arrow{r} \arrow{d} & \mathcal{B}  \\
\mathcal{B}^{\{M_p\}} \arrow{ur} 
\end{tikzcd}
\]
\item[$(iii)$] Suppose that, in addition, $M_p$ satisfies $(M.2)$. Then, for every ultradifferential operator $P(D)$ of class $\{M_p\}$ there is a unique sheaf morphism $P(D): \mathcal{B}^{\{M_p\}} \rightarrow \mathcal{B}^{\{M_p\}}$ that coincides on $\Gamma_c(\R^d, \mathcal{B}^{\{M_p\}}) = \mathcal{E}'^{\{M_p\}}(\R^d)$ with the usual action of $P(D)$ on quasianalytic functionals of class $\{M_p\}$.
\end{enumerate}
\end{proposition}
\begin{proof}
In view of Lemma \ref{extension}, $(i)$ follows from Proposition \ref{support}, $(ii)$ from the fact that the distributional and hyperfunctional support of a distribution coincide, and $(iii)$ holds because 
$$
\operatorname{supp} P(D) f \subseteq \operatorname{supp} f, \qquad f \in \mathcal{E}'^{\{M_p\}}(\R^d).
$$
\end{proof}
\section{Algebras of generalized functions of class $\{M_p\}$}\label{section algebra}
We now introduce differential algebras $\mathcal{G}^{\{M_p\}}(\Omega)$ of generalized functions of class $\{M_p\}$ as quotients of algebras consisting of sequences of ultradifferentiable functions of class $\{M_p\}$ satisfying an appropriate growth condition. Furthermore, we provide a null characterization of the negligible sequences and discuss the sheaf-theoretic properties of the functor $\Omega \rightarrow \mathcal{G}^{\{M_p\}}(\Omega)$. Unless otherwise stated, throughout this section $M_p$ stands for a weight sequence satisfying $(M.1)$, $(M.2)$, $(QA)$, and $(NE)$.

We define the space of $\{M_p\}$-\emph{moderate} sequences as
\begin{align*}
\mathcal{E}^{\{M_p\}}_{\mathcal{M}}(\Omega) = \{ (f_n)_n \in {\mathcal{E}^{\{M_p\}}(\Omega)}^{\N} \, : \, &(\forall K \Subset \Omega) (\forall \lambda > 0) (\exists h > 0) \\
&\sup_{n \in \N} \|f_n\|_{K,h}e^{-M(\lambda n)} < \infty \},    
\end{align*}
and the space of $\{M_p\}$-\emph{negligible} sequences as
\begin{align*}
\mathcal{E}^{\{M_p\}}_{\mathcal{N}}(\Omega) = \{ (f_n)_n \in {\mathcal{E}^{\{M_p\}}(\Omega)}^{\N} \, : \, &(\forall K \Subset \Omega) (\exists \lambda > 0) (\exists h > 0) \\
&\sup_{n \in \N} \|f_n\|_{K,h}e^{M(\lambda n)} < \infty \}.    
\end{align*}

Notice that $\mathcal{E}^{\{M_p\}}_{\mathcal{M}}(\Omega)$ is an algebra under pointwise multiplication of sequences, as follows from $(M.1)$ and $(M.2)$, and that $\mathcal{E}^{\{M_p\}}_{\mathcal{N}}(\Omega)$ is an ideal of $\mathcal{E}^{\{M_p\}}_{\mathcal{M}}(\Omega)$. Hence, we can define the algebra $\mathcal{G}^{\{M_p\}}(\Omega)$ of \emph{generalized functions of class $\{M_p\}$} 
as the factor algebra
$$
\mathcal{G}^{\{M_p\}}(\Omega) = \mathcal{E}^{\{M_p\}}_{\mathcal{M}}(\Omega)/\mathcal{E}^{\{M_p\}}_{\mathcal{N}}(\Omega). 
$$
We denote by $[(f_n)_n]$ the equivalence class of $(f_n)_n \in \mathcal{E}^{\{M_p\}}_{\mathcal{M}}(\Omega)$.
Note that $\mathcal{E}^{\{M_p\}}(\Omega)$ can be regarded as a subalgebra of $\mathcal{G}^{\{M_p\}}(\Omega)$ via the constant embedding
\begin{equation}
\label{eqconstantembedding}
\sigma_\Omega(f): = [(f)_n], \qquad  f \in \mathcal{E}^{\{M_p\}}(\Omega).
\end{equation}
We also remark that $\mathcal{G}^{\{M_p\}}(\Omega)$ can be endowed with a canonical action of ultradifferential operators $P(D)$ of class $\{M_p\}$. In fact, since $P(D)$ acts continuously on $\mathcal{E}^{\{M_p\}}(\Omega)$, we have that  $\mathcal{E}^{\{M_p\}}_{\mathcal{M}}(\Omega)$ and $\mathcal{E}^{\{M_p\}}_{\mathcal{N}}(\Omega)$ are closed under $P(D)$ if we define its action as $P(D)((f_{n})_{n}):=(P(D)f_{n})_{n}$. Consequently, every ultradifferential operator $P(D)$ of class $\{M_p\}$ induces a well-defined linear operator 
$$
P(D):\mathcal{G}^{\{M_p\}}(\Omega)\to \mathcal{G}^{\{M_p\}}(\Omega).
$$

We now provide a null characterization of the ideal $\mathcal{E}^{\{M_p\}}_{\mathcal{N}}(\Omega)$. 

 \begin{lemma} \label{nullchar}
 Let $(f_n)_n \in \mathcal{E}^{\{M_p\}}_{\mathcal{M}}(\Omega)$. Then, $(f_n)_n \in \mathcal{E}^{\{M_p\}}_{\mathcal{N}}(\Omega)$ if and only if 
for every $K \Subset \Omega$ there is $\lambda > 0$ such that
$$
\sup_{n \in \N} \sup_{x \in K}|f_n(x)|e^{M(\lambda n)} < \infty.  
$$
 \end{lemma}
 \begin{proof}
The proof is based on the following multivariable version of Gorny's inequality: Let $K$ and $K'$ be regular compact subsets of $\R^d$ such that $K' \Subset K$. Set $d(K',K^c) = \delta > 0$. Then,
\begin{align}
\max_{|\alpha| = k} \| f^{(\alpha)}\|_{L^\infty(K')} \leq &4 e^{2k}\left(\frac{m}{k}\right)^k \| f \|^{1 - k/m}_{L^\infty(K)} \times \nonumber \\
&\left(\max \left\{d^m\max_{|\alpha| = m} \| f^{(\alpha)}\|_{L^\infty(K)},\frac{\| f \|_{L^\infty(K)}m!}{\delta^m}\right\}\right)^{k/m}
\label{Gornymulti}
\end{align}
for all $f \in C^m(K)$ and $0<k<m$. We prove (\ref{Gornymulti}) below, but let us assume it for the moment and show how the result follows from it.
 
Suppose $(f_n)_n$ satisfies the $0$-th order estimate. Let $K' \Subset \Omega$ be an arbitrary regular compact set. Choose a regular  compact $K$ such that $K' \Subset K \Subset \Omega$. For every $\lambda_1 > 0$ there are $h_1,C_1 > 0$ such that
$$
\| f_n^{(\alpha)}\|_{L^\infty(K)} \leq C_1 h_1^{|\alpha|}M_{\alpha}e^{M(\lambda_1n)}, \qquad \alpha \in \N_0^d, \: n \in \N,
$$
and there are $\lambda_2, C_2 > 0$ such that
$$
\| f_n\|_{L^\infty(K)}\leq C_2e^{-M(\lambda_2n)}, \qquad n \in \N. 
$$
Let $\beta \in \N_0^d$, $\beta \neq 0$. Applying (\ref{Gornymulti}) with $k = |\beta|$ and $m = 2|\beta|$, we obtain
\begin{align*}
&\|f^{(\beta)}_n\|_{L^\infty(K')} 
\\
&\qquad
 \leq4 (2e^{2})^{|\beta|} \| f_n \|^{1/2}_{L^\infty(K)} \left(\max\left\{d^{2|\beta|}\max_{|\alpha| = 2|\beta|} \| f_n^{(\alpha)}\|_{L^\infty(K)},\frac{\| f_n \|_{L^\infty(K)}(2|\beta|)!}{\delta^{2|\beta|}}\right\}\right)^{1/2}.
\end{align*}
Combining this with the above inequalities and taking $\lambda_1:=\lambda_2/H$, one finds by $(M.1)$, $(M.2)$, and $(NE)$ that there are $h, C>0$ such that
\[
\|f^{(\beta)}_n\|_{L^\infty(K')} \leq C h^{|\beta|} M_\beta e^{-M(\lambda_2n/H^2)},\qquad \beta \in \N_0^d,\: n \in \N.
\]
Thus $(f_n)_n \in \mathcal{E}^{\{M_p\}}_{\mathcal{N}}(\Omega)$.

We now show (\ref{Gornymulti}). The one-dimensional Gorny inequality \cite[p.\ 324]{Gorny} states that
\begin{align}
\| g^{(k)}\|_{L^\infty([a,b])} \leq &4 e^{2k}\left(\frac{m}{k}\right)^k \| g \|^{1 - k/m}_{L^\infty([a,b])} \times \nonumber \\
&\left(\max\left\{\| g^{(m)}\|_{L^\infty([a,b])},\frac{\| g \|_{L^\infty([a,b])}2^mm!}{(b-a)^m}\right\}\right)^{k/m},  
\label{Gorny}
\end{align}
for all $g \in C^m([a,b])$ and $0 < k < m$. Denote by $\partial \slash \partial \xi$ the directional derivative in the direction $\xi$, where $\xi \in \R^{d}$ is a unit vector. It is shown in \cite[Thm.\ 2.2]{Chen} that  
\begin{equation}
\max_{|\alpha| = k} \| f^{(\alpha)}\|_{L^\infty(K')} \leq  \sup_{|\xi| = 1} \left \| \frac{\partial^k f}{\partial^k \xi} \right \|_{L^\infty(K')}
\label{directional}
\end{equation}
for all $f \in C^k(K')$, $k \in \N$ -- it is for this inequality that we need the compact $K'$ to be regular. We write $l(x, \xi)$ for the line in $\R^d$ with direction $\xi$ passing through the point $x\in K'$. Define $g_{x, \xi}(t) = f(x + t \xi)$ for $t \in \{ t \in \R \, : \, x + t\xi \in K\}$. The latter set always contains a compact interval $I_{x, \xi} \ni 0$ of length at least $2\delta$. Inequality \eqref{Gorny} therefore implies that
\begin{align*}
\left \| \frac{\partial^k f}{\partial^k \xi} \right \|_{L^\infty(K')} &= \sup_{x \in K'}  | g^{(k)}_{x, \xi}(0)| \leq \sup_{x \in K'}  \| g^{(k)}_{x, \xi}\|_{L^\infty(I_{x, \xi})}
 \\
& \leq \sup_{x \in K'}  4 e^{2k}\left(\frac{m}{k}\right)^k \| g_{x, \xi} \|^{1 - k/m}_{L^\infty(I_{x, \xi})}
\times
\\
&
\quad
\left( \max\left\{\| g_{x, \xi}^{(m)}\|_{L^\infty(I_{x, \xi})},\frac{\| g_{x, \xi} \|_{L^\infty(I_{x, \xi})}m!}{\delta^m}\right\}\right)^{k/m}
 \\
& \leq 4 e^{2k}\left(\frac{m}{k}\right)^k \| f \|^{1 - k/m}_{L^\infty(K)}
\left(\max\left\{\left \| \frac{\partial^m f}{\partial^m \xi}\right\|_{L^\infty(K)},\frac{\| f \|_{L^\infty(K)}m!}{\delta^m}\right\}\right)^{k/m}.
\end{align*}
The result now follows from \eqref{directional} and the fact that 
\begin{align*}
\left \| \frac{\partial^m f}{\partial^m \xi} \right \|_{L^\infty(K)} &= \sup_{x \in K} \left | \sum_{j_1 = 1}^d \cdots \sum_{j_m = 1}^d  \frac{\partial^m f(x)}{\partial x_{j_1}\cdots \partial x_{j_m} } \xi_{j_1} \cdots \xi_{j_m}  \right | 
\leq d^m \sup_{|\alpha| = m} \| f^{(\alpha)}\|_{L^\infty(K)}.
\end{align*}
\end{proof} 
Next, we discuss the sheaf properties of $\mathcal{G}^{\{M_p\}}(\Omega)$. Given an open subset $\Omega'$ of  $\Omega$ and $f = [(f_n)_n] \in \mathcal{G}^{\{M_p\}}(\Omega)$, the restriction of $f$ to $\Omega'$ is defined as
$$
f_{|\Omega'} = [(f_{n|\Omega'})_n] \in \mathcal{G}^{\{M_p\}}(\Omega'). 
$$
Clearly, the assignment $\Omega \rightarrow \mathcal{G}^{\{M_p\}}(\Omega)$ is a presheaf  on $\R^d$. Our aim in the rest of this section is to show that it is in fact a sheaf.  The idea of the proof comes from the theory of hyperfunctions in one dimension: the fact that the hyperfunctions are a sheaf on $\R$ is a direct consequence of the Mittag-Leffler lemma (= solution of the Cousin problem in one dimension) \cite{Morimoto}. Similarly, the solvability of the Cousin problem for the spaces $\mathcal{E}^{\{M_p\}}_{\mathcal{N}}(\Omega)$ would imply that $\mathcal{G}^{\{M_p\}}$ is a sheaf on $\R^d$.  To show the latter, we identify the space $\mathcal{E}^{\{M_p\}}_{\mathcal{N}}(\Omega)$ with a space of vector-valued ultradifferentiable functions of class $\{M_p\}$ with values in an appropriate sequence space and then use Proposition \ref{Cousin}.  We need some preparation: For $\lambda > 0$ we define the following Banach space
$$
s^{M_p,\lambda} = \{ (c_n)_n \in \C^\N \, : \, \sigma_\lambda( (c_n)_n) :=  \sup_{n \in \N} |c_n| e^{M(\lambda n)} < \infty \}.
$$
Set
$$
s^{\{M_p\}} = \varinjlim_{\lambda \rightarrow 0^+}s^{M_p,\lambda},
$$
a $(DFS)$-space. Given a sequence $r_j  \in \mathcal{R}$ we denote by $M_{r_j}$ the associated function of the sequence $M_p\prod_{j= 0}^pr_j$. In \cite[Cor.\ 3.1]{D} we have shown that a sequence $(c_n)_n \in \C^\N$ belongs to $s^{\{M_p\}}$ if and only if 
$$
\sigma_{r_j}((c_n)_n) :=  \sup_{n \in \N} |c_n| e^{M_{r_j}(n)} < \infty
$$
for all  $r_j \in \mathcal{R}$. Moreover, the topology of $s^{\{M_p\}}$ is generated by the system of seminorms $\{\sigma_{r_j} \, : \,  r_j \in \mathcal{R}\}$.
The strong dual of $s^{\{M_p\}}$ is given by
$$
 s_\beta'^{\{M_p\}} = \varprojlim_{\lambda \rightarrow 0^+}s^{M_p,-\lambda},
$$
where
$$
s^{M_p,-\lambda} = \{ (c_n)_n \in \C^\N \, : \, \sigma'_\lambda( (c_n)_n) :=  \sup_{n \in \N} |c_n| e^{-M(\lambda n)} < \infty \}, \qquad \lambda > 0.
$$
\begin{proposition}\label{vectordes}
We have 
$$
\mathcal{E}^{\{M_p\}}_{\mathcal{M}}(\Omega) =  \mathcal{E}^{\{M_p\}}(\Omega;  s'^{\{M_p\}}) \qquad \mbox{and}\qquad\mathcal{E}^{\{M_p\}}_{\mathcal{N}}(\Omega) =  \mathcal{E}^{\{M_p\}}(\Omega;  s^{\{M_p\}}).
$$
\end{proposition}
The proof of Proposition \ref{vectordes} is based on the ensuing lemma.
\begin{lemma}\label{doubleLemma}
Let $(a_{p,n})_{p,n \in \N}$ be a double sequence of positive real numbers. Then,
$$
\sup_{p,n \in \N} \frac{a_{p,n}e^{M(\lambda n)}}{h^p} < \infty
$$
for some $h, \lambda > 0$ if and only if
\begin{equation}
\sup_{p,n \in \N} \frac{a_{p,n}e^{M_{s_j}(n)}}{\prod_{j = 0}^p r_j} < \infty
\label{doubleseq}
\end{equation}
for all $r_j, s_j \in \mathcal{R}$.
\end{lemma}
\begin{proof}
The direct implication is clear. Conversely, suppose \eqref{doubleseq} holds for all $r_j, s_j \in \mathcal{R}$. Equivalently,
$$
\sup_{p,n,q \in \N} \frac{a_{p,n}n^q}{\prod_{j = 0}^p r_j\prod_{j = 0}^q s_j M_q} < \infty,
$$
for all $r_j, s_j \in \mathcal{R}$.  Define
$$
b_{p,q} = \frac{1}{M_q}\sup_{n \in \N} a_{p,n}n^q, \qquad p,q \in \N. 
$$
Then,
$$
\sup_{p,q \in \N} \frac{b_{p,q}}{\prod_{j = 0}^p r_j\prod_{j = 0}^q s_j} < \infty,
$$
for all $r_j, s_j \in \mathcal{R}$. Hence, \cite[Lemma 1]{Pil94} implies that there is $h > 0$ such that
$$
\infty > \sup_{p,q \in \N} \frac{b_{p,q}}{h^{p+q}} = \sup_{p,n,q \in \N} \frac{a_{p,n}n^q}{h^{p+q}M_q} = \sup_{p,n \in \N} \frac{a_{p,n}e^{M(n/h)}}{h^p}.
$$
\end{proof}
\begin{proof}[Proof of Proposition \ref{vectordes}] We only show the second equality. The first one can be shown by a similar argument (but by using \cite[Lemma 3.4]{Komatsu3} instead of Lemma \ref{doubleLemma}). Clearly, the space $\mathcal{E}^{\{M_p\}}(\Omega;  s^{\{M_p\}})$ consists of all $(f_n)_n \in {\mathcal{E}^{\{M_p\}}(\Omega)}^\N$ such that for all $K \Subset \Omega$ and $r_j, s_j \in \mathcal{R}$ it holds that
$$
\sup_{\alpha \in \N_0^d}\sup_{x \in K} \frac{\sigma_{s_j}((f_n^{(\alpha)}(x))_n)}{M_{\alpha}\prod_{j = 0}^{|\alpha|}r_j}  =  \sup_{\alpha \in \N_0^d}\sup_{n \in \N} \frac{\|f_n^{(\alpha)}\|_{L^\infty(K)}e^{M_{s_j}(n)}}{M_{\alpha}\prod_{j = 0}^{|\alpha|}r_j} < \infty.
$$
We find that $(f_n)_n \in \mathcal{E}_{\mathcal{M}}^{\{M_p\}}(\Omega)$ by applying Lemma \ref{doubleLemma} to
$$
a_{p,n} = \max_{|\alpha| \leq p} \frac{\|f_n^{(\alpha)}\|_{L^\infty(K)}}{M_\alpha}, \qquad p,n \in \N.
$$
The converse inclusion can be shown similarly.
\end{proof}
\begin{corollary} \label{projective} We have
 \begin{align}\label{projmod}
\mathcal{E}^{\{M_p\}}_{\mathcal{M}}(\Omega) = \{ (f_n)_n \in  {\mathcal{E}^{\{M_p\}}(\Omega)}^{\N} \, : \, &(\forall K \Subset \Omega) (\forall r_j \in \mathcal{R}) (\exists s_j \in \mathcal{R})  \nonumber \\
 &\sup_{n \in \N}\|f_n\|_{K, r_j}e^{-M_{s_j}(n)} < \infty\},
\end{align}
and
\begin{align}\label{projneg}
\mathcal{E}^{\{M_p\}}_{\mathcal{N}}(\Omega) = \{ (f_n)_n \in  {\mathcal{E}^{\{M_p\}}(\Omega)}^{\N} \, : \, &(\forall K \Subset \Omega) (\forall r_j \in \mathcal{R}) (\forall s_j \in \mathcal{R}) \nonumber \\
 &\sup_{n \in \N}\|f_n\|_{K, r_j}e^{M_{s_j}(n)} < \infty\}.
\end{align}
\end{corollary}
\begin{remark}\label{justification}
Since the structure (choice and order of quantifiers) of  the spaces occurring in \eqref{projmod} and \eqref{projneg} coincides with the structure of the widely accepted definition for spaces of moderate and negligible sequences based on an arbitrary locally convex space \cite{Delcroix}, Corollary \ref{projective} may serve to clarify our definition of $\mathcal{G}^{\{M_p\}}(\Omega)$.
\end{remark}
In \cite[Prop.\ 4.1]{Langenbruch} it is shown that for a weight sequence $M_p$ satisfying $(M.1)$ and $(M.2)'$, the space $s_\beta'^{\{M_p\}}$ satisfies $(\underline{DN})$ if and only if $(M.2)^*$ holds for $M_p$. We have set the ground to show that $\mathcal{G}^{\{M_p\}}$ is a sheaf.
\begin{theorem}\label{sheaf}
Let $M_p$ be a weight sequence satisfying $(M.1)$, $(M.2)$, $(M.2)^*$, $(QA)$, and $(NE)$. The functor $\Omega \rightarrow \mathcal{G}^{\{M_p\}}(\Omega)$ is a sheaf of differential algebras on $\mathbb{R}^{d}$. Furthermore, every ultradifferential operator of class $\{M_p\}$ $P(D): \mathcal{G}^{\{M_p\}} \rightarrow \mathcal{G}^{\{M_p\}}$ is a sheaf morphism.
\end{theorem}
\begin{proof}
It is clear that the presheaf $\mathcal{G}^{\{M_p\}}$ satisfies the first sheaf axiom $(S1)$. We now show that it also satisfies the patching condition $(S2)$. Let $\Omega \subseteq \R^d$ be open and let $\{\Omega_i \, : \, i \in I\}$ be an open covering of $\Omega$. Suppose $f_i = [(f_{i,n})_n] \in \mathcal{G}^{\{M_p\}}(\Omega_i)$ are given such that $f_i = f_j$ on $\Omega_i \cap \Omega_j$, for all $i,j \in I$. This means that there are $(g_{i,j,n})_n \in \mathcal{E}_\mathcal{N}^{\{M_p\}}(\Omega_i \cap \Omega_j)$ such that $g_{i,j,n} = f_{i,n} - f_{j,n}$ on $\Omega_i \cap \Omega_j$, for all $i,j \in I$, $n \in \N$. Since $s^{\{M_p\}}$ is a $(DFS)$-space and $s'^{\{M_p\}}_\beta$ has the property $(\underline{DN})$, Propositions \ref{Cousin} and \ref{vectordes} imply that there are $(g_{i,n})_n \in \mathcal{E}_\mathcal{N}^{\{M_p\}}(\Omega_i )$ such that $g_{i,j,n} = g_{i,n} - g_{j,n}$ on $\Omega_i \cap \Omega_j$, for all $i,j \in I$, $n \in \N$. Hence
$
f_n(x) = f_{i,n}(x) - g_{i,n}(x)$
 if $x \in \Omega_i,
$
is a well-defined function on $\Omega$, for each $n \in \N$, and $(f_n)_n \in \mathcal{E}_\mathcal{M}^{\{M_p\}}(\Omega)$. Set $f = [(f_{n})_n] \in \mathcal{G}^{\{M_p\}}(\Omega)$. Then,
$f_{|\Omega_i} = f_i$ for each $i \in I$.
Finally, it is clear from its definition that $P(D): \mathcal{G}^{\{M_p\}} \rightarrow \mathcal{G}^{\{M_p\}}$ is a sheaf morphism.
\end{proof}

In the next section we establish a very important property of $\mathcal{G}^{\{M_p\}}$, namely, that it is a soft sheaf.
\section{Embedding of the sheaf of infrahyperfunctions} \label{section-embedding}
In this section we shall embed $\mathcal{B}^{\{M_p\}}(\Omega)$ into $\mathcal{G}^{\{M_p\}}(\Omega)$ in such a way that the multiplication of ultradifferentiable functions of class $\{M_p\}$  is preserved. We proceed in various steps. We first construct a support preserving embedding of $\Gamma_c(\R^d, \mathcal{B}^{\{M_p\}}) = \mathcal{E}'^{\{M_p\}}(\R^d)$ into $\Gamma_c(\R^d, \mathcal{G}^{\{M_p\}})$ by means of convolution with a special analytic mollifier sequence that we construct below. Next, we use Lemma \ref{extension} together with the sheaf properties of $\mathcal{B}^{\{M_p\}}$ and $\mathcal{G}^{\{M_p\}}$ to extend this embedding to the whole space $\mathcal{B}^{\{M_p\}}$. This requires proving that $\mathcal{G}^{\{M_p\}}$ is soft. We are very indebted to H\"ormander for his ``hard analysis" type treatment of quasianalytic functionals. A lot of the techniques used in this section are modifications of the ideas from \cite{Hormander}.

We start with a discussion about the mollifier sequences that will be used to embed $\mathcal{E}'^{\{M_p\}}(\R^d)$ into $\Gamma_c(\R^d, \mathcal{G}^{\{M_p\}})$. A sequence 
$(\chi_n)_{n\in \N}$ in $\mathcal{D}(\Omega)$ is called an \emph{analytic cut-off sequence supported in $\Omega$} \cite{Hormander2} if 
\begin{itemize}
\item[$(a)$] $0 \leq \chi_n \leq 1$, $n \in \N$,
\item[$(b)$] $(\chi_n)_{n }$ is a bounded sequence in $\mathcal{D}(\Omega)$,
\item[$(c)$] There is $L \geq 1$ such that
$$ 
\|\chi_n^{(\alpha)}\|_{L^\infty(\R^d)} \leq L(L n)^{|\alpha|}, \qquad  n \in \N,\ |\alpha| \leq n.
$$
\end{itemize}
 We call $(\chi_n)_n$ an analytic cut-off sequence for $K\Subset\Omega$ if
\begin{itemize}
\item[$(d)$] there is an open neighborhood $V$ of $K$ such that $\chi_n \equiv 1$ on $V$ for all $n \in \N$.
\end{itemize}
It is shown in \cite[Thm.\ 1.4.2]{Hormander2} that for every $K \Subset \R^d$ and every open neighborhood $\Omega$ of $K$ there is an analytic cut-off sequence for $K$ supported in $\Omega$. The following two simple lemmas are very useful. We fix the constants in the Fourier transform as
$$
\mathcal{F}(\varphi)(\xi) = \widehat{\varphi}(\xi) = \int_{\R^d} \varphi(x) e^{-ix\xi}\dx.
$$

\begin{lemma}\emph{(Proof of \cite[Thm.\ 3.4]{Hormander})}\label{cutfourier}
Let $M_p$ be a weight sequence satisfying $(M.1)$ and $(NE)$. Let $\Omega$ be a relatively compact open subset in $\R^d$ and let $(\chi_
n)_n$ be an analytic cut-off sequence supported in $\Omega$. Let $h>0$. Then,
$$
|\xi|^n |\widehat{\chi_n\varphi}(\xi)| \leq CL|\Omega|\|\varphi\|_{\overline{\Omega},h} (\sqrt{d}(h+Lk))^{n}M_n, \qquad n \in \N, \: \xi \in \R^d,
$$
for all $\varphi \in \mathcal{E}^{M_p,h}(\overline{\Omega})$, where $C,k > 0$ are chosen in such a way that
$$
p^p \leq Ck^pM_p, \qquad p \in \N_0.
$$
\end{lemma}

\begin{lemma}\emph{(\cite[Lemma 3.3]{Komatsu})}\label{Fourierchar}
Let $h > 0$ and let $\varphi \in L^1(\R^d)$. If
$$
\mu_h(\varphi) := \int_{\R^d} |\widehat{\varphi}(\xi)| e^{M(\xi/h)} \dxi < \infty,
$$
then $\varphi \in C^\infty(\R^d)$ and 
$$
\sup_{x \in \R^d}| \varphi^{(\alpha)}(x)| \leq \frac{\mu_h(\varphi) h^{|\alpha|}M_\alpha}{(2\pi)^d}, \qquad \alpha \in \N_0^d.
$$
\end{lemma}

Let $(\chi_n)_{n}$ be an analytic cut-off sequence for $\overline{B}(0,1)$ consisting of even functions. Define
$$
\theta_n(x) = n^d \mathcal{F}^{-1}(\chi_n)(nx) = \mathcal{F}^{-1}(\chi_n(\cdot / n))(x), \qquad n \in \N.
$$
By Lemma \ref{Fourierchar}, $(\theta_n)_n\in\mathcal{E}^{\{M_p\}}_{\mathcal{M}}(\R^d)$. The sequence $(\theta_n)_{n}$ is called an \emph{$\{M_p\}$-mollifier sequence} if, in addition, the following property holds: for every $c >0$ there are $S,\delta, \gamma > 0$ such that
\begin{equation}
 \sup_{|x| \geq c} |\theta_n^{(\alpha)}(x)| \leq Se^{-M(\delta n)}\gamma
 ^{|\alpha|}M_\alpha, \qquad \alpha \in \N_0^d,\: n \in \N.
\label{mollifier}
\end{equation}
The next lemma establishes the existence of such mollifier sequences; in fact, we provide an explicit construction of a $\{p!\}$-mollifier sequence in its proof.
\begin{lemma}
For every weight sequence $M_p$ satisfying $(M.1)$ and $(NE)$ there exists an $\{M_p\}$-mollifier sequence.
\end{lemma}
\begin{proof} It suffices to consider the case $M_p = p!$. Moreover, if $(\theta_n)_{n}$ is a one-dimensional $\{p!\}$-mollifier sequence, then $(\bm{\theta}_n)_n$ with 
$$
\bm{\theta}_n = \underbrace{\theta_n \otimes \cdots \otimes \theta_n}_{d \text{ times}}
$$
is a $d$-dimensional $\{p!\}$-mollifier sequence, since $(\theta_n)_n\in\mathcal{E}^{\{M_p\}}_{\mathcal{M}}(\R)$. Therefore we may also assume $d =1$. Let $(\kappa_n)_{n }$ be an analytic cut-off sequence for $[-2,2]$ consisting of even functions. We denote by $H$ the characteristic function of $[-1,1]$. Define 
$$
H_n = \left(\frac{n}{2}\right)^n \underbrace{H(n\,\cdot) \ast \cdots \ast H(n\, \cdot)}_{n \text{ times}}, \qquad n \in \N.
$$ 
Notice that $\operatorname*{supp} H_n \subseteq [-1,1]$ and $\int_\R H_n(x) \dx = 1$ for all $n \in \N$. Set $\chi_n = \kappa_n \ast H_n$. Then, $(\chi_n)_{n}$ is an analytic cut-off sequence for $[-1,1]$. We now show that additionally \eqref{mollifier} is satisfied (for $M_p = p!$ and, thus, $M(t) \asymp t$). Notice that 
$$
\theta_n(z) = n\mathcal{F}^{-1}(\chi_n)(nz) = n\mathcal{F}^{-1}(\kappa_n)(nz)(\operatorname*{sinc}(z))^n, \qquad z \in \C,\: n \in \N,
$$
where, as usual, 
$$
\operatorname*{sinc}(z) =\frac{\sin z}{z}=\frac{1}{2}\widehat{H}(z), \qquad z \in \C.
$$
Let $0 < a < \pi$ be arbitrary. Then
$$
|\operatorname*{sinc}(x)| \leq \operatorname*{sinc}(a) =: b < 1, \qquad a \leq |x| \leq \pi.
$$
Furthermore,
$$
|\operatorname*{sinc}(x+iy)| \leq \frac{e^{|y|}}{\pi}, \qquad |x| \geq \pi.
$$
Choose $\mu$ such that $\max\{b,1/\pi\}<\mu<1$. Then there exists $r>0$ such that
$$
|\operatorname*{sinc}(x+iy)| \leq \mu, \qquad |x| \geq a, \: |y|\leq r.
$$
Since the sequence $(\kappa_n)_{n }$ is bounded in $\mathcal{D}(\R)$ there are $C,D > 0$ such that 
$$
|z||\mathcal{F}^{-1}(\kappa_n)(z)| \leq Ce^{D|y|}, \qquad z = x+iy \in \C,\: n \in \N.
$$
Choose $0 < r_0 < r$ so small that $e^{Dr_0}\mu < 1$. Combining the above two inequalities, we obtain
$$
|\theta_n(x+iy)| \leq \frac{C(e^{Dr_0}\mu)^n}{a}, \qquad |x| \geq a, \: |y| \leq r_0, \: n \in \N.
$$
The Cauchy estimates imply that for every $c > 0$ there are $S,\delta, \gamma > 0$ such that
$$
\sup_{|x| \geq c} |\theta_n^{(\alpha)}(x)| \leq Se^{-\delta n}\gamma^{\alpha}\alpha!, \qquad \alpha \in \N_0, \: n \in \N.
$$
\end{proof}
We are ready to embed $\mathcal{E}'^{\{M_p\}}(\R^d)$ into $\Gamma_c(\R^d, \mathcal{G}^{\{M_p\}})$. Fix an analytic cut-off sequence $(\chi_n)_{n}$ for $\overline{B}(0,1)$ consisting of even functions and, say, supported in $B(0,r)$ for some $r > 1$. In addition, suppose that the associated sequence 
$$
\theta_n(x) = n^d \mathcal{F}^{-1}(\chi_n)(nx), \qquad n \in \N,
$$
is an $\{M_p\}$-mollifier sequence (we shall freely use the constant $r$, the constant $L$ in property $(c)$, and the ones appearing in \eqref{mollifier}). Since $\theta_n$ is an entire function, the convolution 
$$
(f\ast \theta_{n})(x) = \langle f(t), \theta_n(x-t)\rangle, \qquad x \in \R^d,
$$
is well-defined for $f \in \mathcal{E}'^{\{M_p\}}(\R^d)$.
\begin{proposition}\label{embedding} Let $M_p$ be a weight sequence satisfying $(M.1)$, $(M.2)$, and $(NE)$. Then, the mapping
$$
\iota_c: \mathcal{E}'^{\{M_p\}}(\R^d) \rightarrow \Gamma_c(\R^d, \mathcal{G}^{\{M_p\}}):  f \mapsto \iota_c(f) = [(f \ast \theta_n)_n],
$$
is a linear embedding. Furthermore, $\operatorname{supp}\iota_c(f) = \operatorname{supp}f$ for all $f  \in \mathcal{E}'^{\{M_p\}}(\R^d)$.
\end{proposition}
We need two lemmas in preparation for the proof of Proposition \ref{embedding}. As customary, we denote by 
$$m(t) := \sum_{m_p \leq t} 1,\qquad t\geq0,$$
the counting function of the sequence $(m_p)_{p \in \N}$. Notice that, by $(M.1)$,
$$
\frac{t^{m(t)}}{M_{m(t)}}=\sup_{p\in\N_0}\frac{t^p}{M_p}=e^{M(t)},
\qquad t \geq 0.
$$
\begin{lemma}\label{lemma1}
Let $\Omega \subseteq \R^d$ be open and let $\Omega'$ be a relatively compact open subset of $\Omega$. Let $\kappa \in \mathcal{D}(\Omega)$ be such that $0 \leq \kappa \leq 1$ and $\kappa \equiv 1$ on a neighborhood of $\overline{\Omega'}$. Then,
$$
 ((\kappa \varphi) \ast \theta_n - \varphi)_n \in  \mathcal{E}^{\{M_p\}}_{\mathcal{N}}(\Omega'),
$$ 
for all $\varphi \in  \mathcal{E}^{\{M_p\}}(\Omega)$. In particular,
$$
 (\kappa \varphi) \ast \theta_n \to \varphi, \qquad \mbox{as } n \to \infty, \mbox{ in } \mathcal{E}^{\{M_p\}}(\Omega').
$$ 
\end{lemma}
\begin{proof}
Fix $\varphi \in \mathcal{E}^{\{M_p\}}(\Omega)$. We claim that 
\begin{equation}
((\kappa\varphi) \ast \theta_n - (\kappa_n \varphi) \ast \theta_n)_n \in \mathcal{E}^{\{M_p\}}_\mathcal{N}(\Omega') 
\label{claim2}
\end{equation}
for any bounded sequence $(\kappa_n)_{n}$ in $\mathcal{D}(\Omega)$ such that $0 \leq \kappa_n \leq 1$ and for which there is a neighborhood $U$ of $\overline{\Omega'}$ such that $\kappa_n  \equiv 1$ on $U$ for all $n \in \N$. Before proving the claim, we show how it implies the result. Let $(\rho_p)_{p}$ be an analytic cut-off sequence for $\overline{\Omega'}$ supported in $\Omega$. Lemma \ref{cutfourier} implies that there are $C,h > 0$ such that
$$
|\xi|^p|\widehat{\rho_p\varphi}(\xi)| \leq Ch^{p}M_p, \qquad p \in \N, \xi \in \R^d.
$$
Set $p_n = m(n/(Hh)) +d +1$ and $\kappa_n = \rho_{p_n}$, $n \in \N$. By the claim it suffices to show that
$$
((\kappa_n \varphi) \ast \theta_n - \varphi)_n \in \mathcal{E}^{\{M_p\}}_\mathcal{N}(\Omega'). 
$$
For $K \Subset \Omega'$ we have
\begin{align*}
\sup_{x \in K}|((\kappa_n \varphi) \ast \theta_n)(x) - \varphi(x)| &\leq \sup_{x \in K}|((\kappa_n \varphi) \ast \theta_n)(x) - (\kappa_n\varphi)(x)|  \\
& \leq \frac{1}{(2\pi)^d}\int_{\R^d} |\widehat{\kappa_n\varphi}(\xi)|(1 -\chi_n(\xi/n))\dxi \\
& \leq \frac{AC(Hh)^{d+1}M_{d+1}}{(2\pi)^d} \int_{|\xi| \geq n} \frac{(Hh)^{m(n/(Hh))}M_{m(n/(Hh))}}{|\xi|^{m(n/(Hh)) +d +1}} \dxi \\
& \leq D e^{-M(n/(Hh))},
\end{align*}
where 
$$
D = \frac{AC(Hh)^{d+1}M_{d+1}}{(2\pi)^d} \int_{|\xi| \geq 1} \frac{1}{|\xi|^{d +1}} \dxi < \infty,
$$
whence the result follows (cf.\ Lemma \ref{nullchar}). We now show \eqref{claim2}. We first prove that
$$
((\kappa_n \varphi) \ast \theta_n)_n \in \mathcal{E}^{\{M_p\}}_\mathcal{M}(\R^d). 
$$
Since the sequence $(\kappa_n \varphi)_n$ is bounded in $\mathcal{D}(\R^d)$ there is $C > 0$ such that
$$
|\widehat{\kappa_n \varphi}(\xi)| \leq \frac{C}{(1+ |\xi|)^{d+1}}, \qquad \xi \in \R^d,\: n \in \N.
$$
For $\lambda > 0$ we have
$$
\mu_{1/\lambda}((\kappa_n \varphi) \ast \theta_n) = \int_{\R^d} |\widehat{\kappa_n\varphi}(\xi)| \chi_n(\xi/n)e^{M(\lambda \xi)}\dxi \leq Ce^{M(\lambda r n)}\int_{\R^d}\frac{1}{(1+ |\xi|)^{d+1}} \dxi,
$$
and therefore the sequence is moderate by Lemma \ref{Fourierchar}. Hence, by Lemma \ref{nullchar}, it suffices to show that for every $K \Subset \Omega'$ there is $\lambda > 0$ such that
$$
\sup_{n \in \N}\sup_{x \in K}|((\kappa\varphi) \ast \theta_n)(x) - ((\kappa_n \varphi) \ast \theta_n)(x)|e^{M(\lambda n)} < \infty.
$$
Choose $c > 0$ such that $\kappa - \kappa_n \equiv 0$ on $K + B(0,c)$ and let $K_0 \Subset \Omega$ be such that $\operatorname*{supp}(\kappa - \kappa_n) \subseteq K_0$ for all $n \in \N$. Then,
\begin{align*}
\sup_{x \in K}|((\kappa\varphi) \ast \theta_n)(x) - ((\kappa_n \varphi) \ast \theta_n)(x)| &\leq \sup_{x \in K}\left|\int_{\R^d}(\kappa(x-t) - \kappa_n(x-t))\varphi(x-t)  \theta_n(t) \dt \right| \\
&\leq 2S\|\varphi\|_{L^\infty(K_0)}  |K_0| e^{-M(\delta n)}.
\end{align*}
\end{proof}
For $h > 0$, we write $\mathcal{E}^{M_p,h}_{\infty}(\R^d)$ for the Banach space consisting of all  $\varphi \in C^\infty(\R^d)$ such that
$$
\sup_{\alpha \in \N_0^d}\sup_{x \in \R^d} \frac{|\varphi^{(\alpha)}(x)|}{h^{|\alpha|}M_{\alpha}} < \infty.
$$
\begin{lemma}\emph{(cf.\ first part of \cite[Thm. 3.4]{Hormander})} \label{lemma2}Let $h > 0$ and let $\Omega$ be a relatively compact open subset of $\R^d$. Suppose that $(\kappa_p)_{p }$ is an analytic cut-off sequence supported in $\Omega$ and that $(\psi_n)_{n }$ is a uniformly bounded sequence of continuous functions on $\R^d$ such that $\operatorname*{supp}(\psi_1) \subseteq B(0,r)$ and 
$$
\operatorname*{supp}(\psi_n) \subseteq B(0,rn) \backslash \overline{B}(0,n-1), \qquad n \geq 2. 
$$ 
Then, there are a sequence $(p_n)_{n }$ of natural numbers and $k > 0$ such that
$$
R(\varphi) := \sum_{n=1}^\infty(\kappa_{p_n}\varphi) \ast \mathcal{F}^{-1}(\psi_n) \in \mathcal{E}^{M_p,k}_{\infty}(\R^d), \qquad \varphi \in \mathcal{E}^{M_p,h}(\overline{\Omega}).
$$
Moreover, the convergence of the series $R(\varphi)$ holds in the topology of $\mathcal{E}^{M_p,k}_{\infty}(\R^d)$ and  the mapping $R:  \mathcal{E}^{M_p,h}(\overline{\Omega}) \rightarrow \mathcal{E}^{M_p,k}_{\infty}(\R^d)$ is continuous.
\end{lemma}
\begin{proof} By Lemma \ref{cutfourier} there are $C, k_0 >0$ such that
$$
|\xi|^p|\widehat{\kappa_p\varphi}(\xi)| \leq C\|\varphi\|_{\overline{\Omega},h} k_0^{p}M_p, \qquad p \in \N, \: \xi \in \R^d,
$$
for all $\varphi \in \mathcal{E}^{M_p,h}(\overline{\Omega})$. For $p = m(t)$, $t \geq 0$, and $k_0 t \leq |\xi| \leq 2r k_0t$, we obtain
$$
|\widehat{\kappa_{m(t)}\varphi}(\xi)| \leq C\|\varphi\|_{\overline{\Omega},h}e^{-M(\xi/(2k_0r))}.
$$
Set $p_n = \max\{1,m((n-1)/k_0)\}$, $n \in \N$. Hence,
$$
|\widehat{\kappa_{p_n}\varphi}(\xi)| \leq C\|\varphi\|_{\overline{\Omega},h}e^{-M(\xi/(2k_0r))}, \qquad (n-1) \leq |\xi| \leq rn,\:  n\geq k_0 + 1.
$$
Choose $C' > 0$ such that $\|\psi_n \|_{L^\infty(\R^d)} \leq C'$ for all $n \in \N$. Then, for $k = 2H^2k_0r$, we have
\begin{align*}
\mu_k\left( \sum_{n \geq k_0 + 1}^\infty (\kappa_{p_n}\varphi) \ast  \mathcal{F}^{-1}(\psi_n)  \right) &\leq \sum_{n \geq k_0 +1}^\infty \int_{\R^d} |\psi_n(\xi)| |\widehat{\kappa_{p_n}\varphi}(\xi)| e^{M(\xi/k)}\dxi \\ 
&\leq C'C\|\varphi\|_{\overline{\Omega},h} \sum_{n \geq k_0 +1}^\infty \int_{(n-1) \leq |\xi| \leq rn}   e^{M(\xi/k)-M(\xi/(2k_0r))}\dxi \\
&\leq D\|\varphi\|_{\overline{\Omega},h},
\end{align*}
where
$$
D = A^2 C'C\sum_{n=k_0}^\infty e^{-M(n/k)}\int_{\R^d}e^{-M(\xi/(2Hk_0r))}\dxi < \infty.
$$
We also have
$$
\mu_k((\kappa_{p_n}\varphi) \ast  \mathcal{F}^{-1}(\psi_n)) \leq C'e^{M(rk_0/k)}|B(0,rk_0)| |\Omega|\|\varphi\|_{L^\infty(\Omega)},
$$
for $n \leq k_0$. The result now follows from Lemma \ref{Fourierchar}.
\end{proof}
\begin{proof}[Proof of Proposition \ref{embedding}]
STEP I: \emph{$(f \ast \theta_n)_n \in \mathcal{E}^{\{M_p\}}_\mathcal{M}( \R^d)$ for all $f \in \mathcal{E}'^{\{M_p\}}( \R^d)$}: Let $\lambda > 0$ be arbitrary. Choose $C > 0$ such that
$$
|\widehat{f}(\xi)| \leq C e^{M(\lambda\xi/H)}, \qquad \xi \in \R^d.
$$
We have
\begin{align*}
\mu_{1/\lambda}(f \ast \theta_n) &= \int_{\R^d} |\widehat{f}(\xi)| \chi_n(\xi/n) e^{M(\lambda \xi)} \dxi
\leq C \int_{|\xi| \leq rn} e^{M(\lambda \xi/H) + M(\lambda \xi)} \dxi  \leq De^{M(H\lambda r n)},
\end{align*}
where
$$
D = A^2C\int_{\R^d} e^{-M(\lambda \xi / H)}\,d\xi < \infty.
$$
The result follows from Lemma \ref{Fourierchar}.

STEP II: \emph{$\operatorname{supp}\iota_c(f) \subseteq \operatorname{supp}f$ for all $f  \in \mathcal{E}'^{\{M_p\}}(\R^d)$}. Let $K \Subset \R^d \backslash \operatorname*{supp} f$ be arbitrary. Choose $K'\Subset\R^d$ such that a neighborhood of $\operatorname*{supp} f$ is contained in $K'$ and $K'\cap K=\emptyset$. Set $c:= d(K,K') > 0$.  The continuity of $f$ implies that for each $h>0$ there is $C>0$ such that 
$$
\sup_{x \in K}|(f \ast \theta_n)(x)| \leq C \|\theta_n\|_{K-K', h}.
$$
By Property \eqref{mollifier} we have that
$$ 
\| \theta_n \|_{K-K', \gamma} = \sup_{\alpha \in \N_0^d}\sup_{x \in K- K'} \frac{|\theta_n^{(\alpha)}(x)|}{\gamma^{|\alpha|}M_{\alpha}} \leq Se^{-M(\delta n)}.
$$
We then obtain that $\iota_{c}(f)$ vanishes on $\R^d \backslash \operatorname*{supp} f$ from Lemma \ref{nullchar}.

STEP III: \emph{$\operatorname{supp}f \subseteq \operatorname{supp}\iota_c(f)$ for all $f  \in \mathcal{E}'^{\{M_p\}}(\R^d)$}. Set $K = \operatorname{supp}\iota_c(f)$. We need to show that $f \in \mathcal{E}'^{\{M_p\}}(\Omega)$ for every open set $\Omega \Supset K$. Fix such a set $\Omega$. By Lemma \ref{systemnorms} and \cite[Lemma 2.3]{Prangoski} there is a weight sequence $N_p$ satisfying $(M.1)$, $(M.2)$, and $M_p \prec N_p$ such that $f \in  \mathcal{E}'^{\{N_p\}}(\R^d)$.  Choose $\rho \in \mathcal{D}(\R^d)$ such that  $0 \leq \rho \leq 1$ and $\rho \equiv 1$ on a neighborhood of $\operatorname*{supp} f$. Employing Lemma \ref{lemma1}, we deduce that
$$
\langle f , \varphi \rangle = \lim_{n \to \infty} \langle f ,  (\rho \varphi) \ast \theta_n \rangle =  \lim_{n \to \infty} \int_{\R^d} (f \ast \theta_n)(x) \rho(x) \varphi(x) \dx, \qquad \varphi \in \mathcal{E}^{\{M_p\}}(\R^d).
$$ 
Next, choose $\kappa \in \mathcal{D}(\Omega)$ such that  $0 \leq \kappa \leq 1$ and $\kappa \equiv 1$ on a neighborhood of $K$. By STEP II we have that $\kappa - \rho \equiv 0$ on a neighborhood of $K$. Therefore, $\iota_c(f)_{| \R^d \backslash K} = 0$ implies that $f \ast \theta_n \rightarrow 0$, as $ n \to \infty$, uniformly on $\operatorname*{supp}(\kappa - \rho)$. Hence,
\begin{equation}
\langle f , \varphi \rangle = \lim_{n \to \infty}\int_{\R^d} (f \ast \theta_n)(x) \kappa(x) \varphi(x) \dx, \qquad \varphi \in \mathcal{E}^{\{M_p\}}(\R^d).
\label{reduction1}
\end{equation} 
We now invoke Lemma \ref{lemma2}. Set $\psi_1 = \chi_1$ and
$$
\psi_n = \chi_n\left(\frac{\cdot}{n}\right) - \chi_{n-1}\left(\frac{\cdot}{n-1}\right), \qquad n \geq 2.
$$ 
Clearly, the sequence $(\psi_n)_{n}$ satisfies the requirements of Lemma \ref{lemma2}. Choose a relatively compact open subset $\Omega'$ such that $K \Subset \Omega' \Subset \Omega$ and an analytic cut-off sequence $(\kappa_p)_{p }$ for $K$ supported in $\Omega'$. According to Lemma \ref{lemma2} (applied to the weight sequence $N_p$), there are a sequence $(p_n)_{n }$ and $k > 0$ such that the mapping
$$
R: \mathcal{E}^{N_p,1}(\overline{\Omega'}) \rightarrow \mathcal{E}^{N_p,k}_{\infty}(\R^d): \varphi \rightarrow R(\varphi) = \sum_{n=1}^\infty  (\kappa_{p_n}\varphi) \ast \mathcal{F}^{-1}({\psi}_n) 
$$
is continuous. Consider the following continuous inclusion mappings 
$$
\iota_1: \mathcal{E}^{\{M_p\}}(\Omega) \rightarrow \mathcal{E}^{N_p,1}(\overline{\Omega'}), \qquad \iota_2: \mathcal{E}^{N_p,k}_\infty(\R^d) \rightarrow \mathcal{E}^{\{N_p\}}(\R^d),
$$
and set $T = \iota_2 \circ R \circ \iota_1 : \mathcal{E}^{\{M_p\}}(\Omega) \rightarrow  \mathcal{E}^{\{N_p\}}(\R^d)$. Equality \eqref{reduction1} gives us
\begin{align*}
\langle f , \varphi \rangle =  &\int_{\R^d} \sum_{n = 1}^\infty(f \ast \mathcal{F}^{-1}(\psi_n))(x) (\kappa(x) - \kappa_{p_n}(x)) \varphi(x) \dx \, + \\
&  \sum_{n = 1}^\infty \int_{\R^d} (f \ast \mathcal{F}^{-1}(\psi_n))(x)\kappa_{p_n}(x) \varphi(x) \dx, 
\end{align*}
for all $\varphi \in \mathcal{E}^{\{M_p\}}(\R^d)$. Since $(\kappa - \kappa_{p_n})_{n}$ is a bounded sequence in $\mathcal{D}(\Omega \backslash K)$, the assumption $\iota_c(f)_{| \R^d \backslash K} = 0$ yields 
$$
g = \sum_{n = 1}^\infty(f \ast \mathcal{F}^{-1}(\psi_n)) (\kappa - \kappa_{p_n}) \in \mathcal{D}(\Omega).
$$
For the second term, we have 
$$
\sum_{n = 1}^\infty \int_{\R^d} (f \ast \mathcal{F}^{-1}(\psi_n))(x)\kappa_{p_n}(x) \varphi(x) \dx = \sum_{n = 1}^\infty \langle f,  \mathcal{F}^{-1}(\psi_n) \ast (\kappa_{p_n}\varphi) \rangle = \langle f, T(\varphi) \rangle, 
$$
for all $\varphi \in \mathcal{E}^{\{M_p\}}(\R^d)$. Hence $f = g + {}^tT(f) \in \mathcal{E}'^{\{M_p\}}(\Omega)$.
\end{proof}
\begin{proposition}\label{soft}
Let $M_p$ be a weight sequence satisfying $(M.1)$, $(M.2)$, $(M.2)^*$, $(QA)$, and $(NE)$. Then, the sheaf $\mathcal{G}^{\{M_p\}}$ is soft.
\end{proposition}
\begin{proof}
Let $K \Subset \R^d$ and $K \Subset \Omega$, $\Omega$ open, be arbitrary. Choose a relatively compact open set $\Omega'$ such that $K \Subset \Omega' \Subset \Omega$. It suffices to show that for every $f = [(f_n)_n] \in \mathcal{G}^{\{M_p\}}(\Omega)$ there is $g = [(g_n)_n] \in \mathcal{G}^{\{M_p\}}(\R^d)$ such that $g = f$ on $\Omega'$. Let $(\kappa_p)_p$ be an analytic cut-off sequence for $\overline{\Omega'}$ supported in $\Omega$. Lemma \ref{cutfourier} and $(f_n)_n \in \mathcal{E}^{\{M_p\}}_\mathcal{M}(\Omega)$ imply that there are $C,h > 0$ such that
$$
|\xi|^p|\widehat{\kappa_pf_n}(\xi)| \leq Ce^{M(n)}h^pM_p, \qquad \xi \in \R^d,\: p,n \in \N.
$$
Let $\psi \in \mathcal{D}(B(0,2))$ be such that $0 \leq \psi \leq 1$ and $\psi \equiv 1$ on $\overline{B}(0,1)$. Set $a = H^2h$ and 
$$
\psi_n = \psi\left( \frac{\cdot}{an}\right), \qquad n \in \N,
$$
and define
$$
g_n = (\kappa_{p_n}f_n) \ast \mathcal{F}^{-1}(\psi_n), \qquad n \in \N,
$$
where $p_n = m(Hn) + d + 1$. We first show that $(g_n)_n \in \mathcal{E}^{\{M_p\}}_\mathcal{M}(\R^d)$. 
Since the sequence $(\kappa_p)_p$ is bounded in $\mathcal{D}(\Omega)$ and $(f_n)_n \in \mathcal{E}^{\{M_p\}}_\mathcal{M}(\Omega)$, we have
$$
|\widehat{\kappa_pf_n}(\xi)| \leq \frac{D_{\lambda} e^{M(\lambda n)}}{(1+|\xi|)^{d+1}}, \qquad \xi \in \R^d, \: p,n \in \N,
$$
for every $\lambda > 0$ and suitable $D_\lambda > 0$. Hence
$$
\mu_{1/\lambda}(g_n) = \int_{\R^d} |\widehat{\kappa_{p_n}f_n}(\xi)| \psi_n(\xi)e^{M(\lambda \xi)}\dxi \leq AD_{2a\lambda}e^{M(2Ha\lambda n)}\int_{\R^d} \frac{1}{(1+|\xi|)^{d+1}}\dxi
$$
for all $\lambda >0$.  The sequence $(g_n)_n$ is therefore moderate by Lemma \ref{Fourierchar}. We still need to show that
$$
(g_n - f_n)_n \in \mathcal{E}^{\{M_p\}}_\mathcal{N}(\Omega'). 
$$
We verify this via Lemma \ref{nullchar}. For any $K \Subset \Omega'$, we have
\begin{align*}
\sup_{x \in K}|g_n(x) - f_n(x)| & \leq \sup_{x \in K}|((\kappa_{p_n} f_n) \ast \mathcal{F}^{-1}(\psi_n))(x) - (\kappa_{p_n}f_n)(x)|  \\
& \leq \frac{1}{(2\pi)^d}\int_{\R^d} |\widehat{\kappa_{p_n}{f_n}}(\xi)|(1 -\psi_n(\xi))\dxi \\
& \leq \frac{AC(Hh)^{d+1}M_{d+1}e^{M(n)}}{(2\pi)^d} \int_{|\xi| \geq an} \frac{(Hh)^{m(Hn)}M_{m(Hn)}}{|\xi|^{m(Hn) +d +1}} \dxi \\
& \leq D e^{-M(n)},
\end{align*}
where 
$$
D = \frac{A^2C(Hh)^{d+1}M_{d+1}}{(2\pi)^d} \int_{|\xi| \geq a} \frac{1}{|\xi|^{d +1}} \dxi < \infty,
$$
whence the result follows. 
\end{proof}
We have completed all necessary work to prove our main theorem. Recall that $\sigma_{\Omega}:\mathcal{E}^{\{M_p\}}(\Omega) \to \mathcal{G}^{\{M_p\}}(\Omega)$ stands for the constant embedding \eqref{eqconstantembedding}.
\begin{theorem}\label{main-qa}
Let $M_p$ be a weight sequence satisfying $(M.1)$, $(M.2)$, $(M.2)^*$, $(QA)$, and $(NE)$. Then, there is a unique injective sheaf morphism $\iota: \mathcal{B}^{\{M_p\}} \rightarrow \mathcal{G}^{\{M_p\}}$ such that for each open set $\Omega \subseteq \R^d$ the following properties hold:
\begin{enumerate}
\item[$(i)$] $\iota_{\Omega}(f) = \iota_c(f)$ for all $f \in \mathcal{E}'^{\{M_p\}}(\Omega)$.
\item[$(ii)$] For all ultradifferential operators $P(D)$ of class ${\{M_p\}}$ we have
$$
P(D)\iota_\Omega(f) = \iota_\Omega(P(D)f), \qquad f \in \mathcal{B}^{\{M_p\}}(\Omega).
$$
\item[$(iii)$] $\iota_{\Omega}(f) = \sigma_\Omega(f)$ for all $f \in \mathcal{E}^{\{M_p\}}(\Omega)$. Consequently, 
$$
\iota_{\Omega}(fg) = \iota_\Omega(f)\iota_\Omega(g),\qquad f,g  \in \mathcal{E}^{\{M_p\}}(\Omega).
$$ 
\end{enumerate}
\end{theorem}
\begin{proof} Since $\mathcal{B}^{\{M_p\}}$  and $\mathcal{G}^{\{M_p\}}$ are soft sheaves (Propositions \ref{constructioninfra} and \ref{soft}), the existence and uniqueness of a sheaf embedding $\iota$ satisfying properties $(i)$ and $(ii)$ is clear from Lemma \ref{extension} and Proposition \ref{embedding}. We now show that property $(iii)$ is also satisfied. It suffices to show that $\iota_{\Omega}(f)_{|\Omega'} = \sigma_\Omega(f)_{|\Omega'}$ for all open sets $\Omega' \Subset \Omega$. Fix such a set $\Omega'$ and choose $\kappa \in \mathcal{D}(\Omega)$ such that $0 \leq \kappa \leq 1$ and $\kappa \equiv 1$ on a neighborhood of $\overline{\Omega'}$. Notice that $\iota_{\Omega}(f)_{|\Omega'} = \iota_{\Omega'}(f_{|\Omega'}) = \iota_{\Omega'}((\kappa f)_{|\Omega'})= \iota_{\Omega}(\kappa f)_{|\Omega'} = \iota_c(\kappa f) _{|\Omega'} $. Hence, it suffices to show that 
$$
((\kappa f) \ast \theta_n - f)_n \in \mathcal{E}^{\{M_p\}}_{\mathcal{N}}(\Omega'),
$$
but this has already been proved in Lemma \ref{lemma1}.
\end{proof}
We end this article with a remark concerning the optimality of the above embedding. It can be viewed as an analogue of Schwartz's impossibility result \cite{Schwartz} in the setting of infrahyperfunctions, which states that, under some natural assumptions, the property $(iii)$ from Theorem \ref{main-qa} \emph{cannot} be improved to also preserve the multiplication of all ultradifferentiable functions from a class with lower regularity than $\mathcal{E}^{\{M_p\}}(\Omega)$.
\begin{remark}
Let $N_p$ be another weight sequence satisfying $(M.1)$ and $(M.2)'$. When embedding $\mathcal{B}^{\{M_p\}}(\Omega)$ into an associative and commutative algebra $(\mathcal{A}^{\{M_p\},\{N_p\}}(\Omega), +, \circ)$, the following requirements seem to be natural:
\begin{enumerate}
\item[$(P.1)$] $\mathcal{B}^{\{M_p\}}(\Omega)$ is linearly embedded into $\mathcal{A}^{\{M_p\},\{N_p\}}(\Omega)$ and $f(x) \equiv 1$ is the unity in $\mathcal{A}^{\{M_p\},\{N_p\}}(\Omega)$.
\item[$(P.2)$] For each ultradifferential operator $P(D)$ of class $\{M_p\}$  there is a linear operator $P(D):  \mathcal{A}^{\{M_p\},\{N_p\}}(\Omega) \rightarrow \mathcal{A}^{\{M_p\},\{N_p\}}(\Omega)$ satisfying the following generalized Leibniz' rule
 $$
 P(D)( q \circ f ) = \sum_{\beta \leq \deg q} \frac{1}{\beta!} D^\beta q \circ (D^\beta P)(D)f,  
 $$
for every $f \in \mathcal{A}^{\{M_p\},\{N_p\}}(\Omega)$ and every polynomial $q$. Moreover, $P(D)_{|\mathcal{B}^{\{M_p\}}(\Omega)}$ coincides with the usual action of $P(D)$ on infrahyperfunctions of class $\{M_p\}$.
\item[$(P.3)$] $\circ_{|\mathcal{E}^{\{N_p\}}(\Omega) \times \mathcal{E}^{\{N_p\}}(\Omega)}$ coincides with the pointwise product of functions.
\end{enumerate}
\smallskip

The ensuing result imposes a limitation on the possibility of constructing such an algebra if $M_p \prec N_p$ (implying that $\mathcal{E}^{\{M_p\}}(\Omega)\subsetneq \mathcal{E}^{\{N_p\}}(\Omega)$). On the other hand, 
our differential algebra $\mathcal{G}^{\{M_p\}}(\Omega)$ satisfies all the properties $(P.1)$-$(P.3)$ for $M_p = N_p$ and therefore Theorem \ref{main-qa} is optimal in this sense.

\begin{theorem} \label{impossibility} Let $M_p$ be a weight sequence satisfying $(M.1)$, $(M.2)$, $(QA)$, and $(NE)$ and let $N_p$ be a weight sequence satisfying $(M.1)$ and $(M.2)'$. Suppose that $M_p \prec N_p$. Then, there is no associative and commutative algebra $\mathcal{A}^{\{M_p\},\{N_p\}}(\Omega)$ satisfying $(P.1)$-$(P.3)$. 
\end{theorem}
\begin{proof} The proof is basically the same as that of the corresponding result for non-quasianalytic ultradistributions \cite[Thm. 3.1]{DVV}, but we include it for the sake of completeness. Suppose $\mathcal{A}^{\{M_p\},\{N_p\}}(\Omega)$ is such an algebra. We have that $q \circ P(D)g =  q P(D)g$ for every ultradifferential operator $P(D)$ of class $\{M_p\}$, $g \in  \mathcal{E}^{\{N_p\}}(\Omega)$ and polynomial $q$; this follows by induction on the degree of $q$. Assume for simplicity that $0 \in \Omega$. Write $H(x)= H(x_1, \ldots, x_d): =  H(x_1) \otimes \cdots \otimes H(x_d)$, where $H(x_j)$ is the Heaviside function (characteristic function of the positive half-axis), and  $\operatorname*{p.v.}(x^{-1}) =  \operatorname*{p.v.}(x_1^{-1}) \otimes \cdots \otimes  \operatorname*{p.v.}(x_d^{-1})$, where $\operatorname*{p.v.}(x_j^{-1})$ is the principle value regularization of the function $x_j^{-1}$.  Let $f$ be either $H(x)$ or $\operatorname{p.v. }(x^{-1})$. By employing the global structural theorem for infrahyperfunctions of class $\{M_p\}$ \cite{Takiguchi}, we find $g \in \mathcal{E}^{\{N_p\}}(\R^d)$ and an ultradifferential operator $P(D)$ of class $\{M_p\}$ such that $P(D)g = f$. Set
$$
\boldsymbol{\partial} =  \frac{\partial^d}{\partial x_1 \cdots \partial x_d} \quad \mbox{and} \quad q(x)=x_1x_2\cdots x_d.
$$
The observation made at the beginning of the proof now yields $q\circ \boldsymbol{\partial} H = (q \boldsymbol\partial H)$ and $q \circ \operatorname{p.v. }(x^{-1}) = q \operatorname{p.v. }(x^{-1})$. Since $q \boldsymbol{\partial} H = 0$ and $ q \operatorname{p.v. }(x^{-1}) = 1$ in  $\mathcal{B}^{\{M_p\}}(\Omega)$, we obtain
$$
 \boldsymbol{\partial} H= \boldsymbol{\partial} H \circ ( q \circ \operatorname*{p.v.}(x^{-1})) =  (  \boldsymbol{\partial} H \circ  q) \circ \operatorname*{p.v.}(x^{-1}) = 0,
$$
contradicting $\boldsymbol{\partial} H = \delta \neq 0$ in $\mathcal{B}^{\{M_p\}}(\Omega)$ and the injectivity of $\mathcal{B}^{\{M_p\}}(\Omega)\to \mathcal{A}^{\{M_p\},\{N_p\}}(\Omega)$.
\end{proof}
\end{remark}

\end{document}